\documentclass[12pt, a4paper]{amsart}
\usepackage{amscd, latexsym}
\usepackage{amsmath, amsthm, amstext, amsbsy, amssymb}
\usepackage{mathrsfs}

\usepackage[backend=biber, style=alphabetic, sorting=anyt, nohashothers=true, maxbibnames=99]{biblatex}
\usepackage{xcolor}                 % coloring text 
\usepackage[shortlabels]{enumitem} 

\usepackage{hyperref} 				% references
\hypersetup{						% hyperlinks 
	colorlinks=true,
	linkcolor=blue,
	citecolor=red,
}

\allowdisplaybreaks		   			% breaks in formulas

\newtheorem{lem}{Lemma}
\newtheorem{thrm}{Theorem}
\newtheorem{thrmrestate}{Theorem}
\newtheorem*{thrm*}{Theorem}
\newtheorem{thrmm}{Theorem}
 
\newtheorem{crl}{Corollary}

\theoremstyle{definition}  %italics in theorems
\newtheorem{definition}{Definition}

\renewcommand{\le}{\leqslant}
\renewcommand{\ge}{\geqslant}

\let\intt\int
\renewcommand{\int}{\intt\limits}

\newcommand{\C}{\mathbb{C}} 						% Complex numbers
\newcommand{\D}{\mathbb{D}} 						% Disk
\newcommand{\N}{\mathbb{N}}	                	    % Naturals
\newcommand{\TT}{\mathbb{T}}	 					% Circle
\renewcommand{\H}{\mathcal{H}}                      % Hilbert transform
\newcommand{\M}{\mathcal{M}}                        % Mupltipliers

\newcommand{\B}{\mathbb{B}}                         % Unit Ball

\renewcommand{\a}{\alpha}									% alpha | A
\renewcommand{\b}{\beta}									% beta | B
\newcommand{\de}{\delta}									% delta
\newcommand{\e}{\varepsilon}								% epsilon | E
\newcommand{\z}{\zeta}										% zeta | Z
\renewcommand{\k}{\kappa}									% kappa | K
\renewcommand{\l}{\lambda}									% lambda
\let\originalnu\nu											% <fixing nu>
\renewcommand{\nu}{\originalnu} 					        % nu | N
\newcommand{\f}{\varphi}									% <another phi>
\renewcommand{\o}{\omega}									% omega 
\DeclareMathOperator{\Cl}{Cl}							    % closure
\renewcommand{\Re}{\operatorname{Re}}		    	% normal real part
\DeclareMathOperator{\Span}{span}							% span
\DeclareMathOperator{\id}{id}								% identity map
\DeclareMathOperator{\Mult}{Mult}

\title[The isomorphism problem for analytic discs]{The isomorphism problem for analytic discs with self-crossings on the boundary}
\author{Mikhail Mironov}
\address{
\newline \phantom{x}\,\, Mikhail Mironov,
\newline Department of Mathematics, Technion - Israel Institute of Technology, Haifa, Israel,
\newline Univ Gustave Eiffel, Univ Paris Est Creteil, CNRS, LAMA UMR8050 F-77447 Marne-la-Vallée, France
\newline {\tt mikhailm@campus.technion.ac.il}
\smallskip
}

\begin{document}

\begin{abstract}
    Suppose $V$ is the unit disc $\mathbb{D}$ embedded in the $d$-dimensional unit ball $\mathbb{B}_d$ and attached to the unit sphere. Consider the space $\mathcal{H}_V$, the restriction of the Drury-Arveson space to the variety $V$, and its multiplier algebra $\M_V = \Mult(\H_V)$.
    The isomorphism problem is the following: Is $V_1 \cong V_2$ equivalent to $\M_{V_1} \cong \M_{V_2}$?

    A theorem of Alpay, Putinar and Vinnikov states that for $V$ without self-crossings on the boundary $\M_V$ is the space of bounded analytic functions on $V$. We consider what happens when there are self-crossings on the boundary and prove that if $\M_{V_1} \cong \M_{V_2}$ algebraically, then $V_1$ and $V_2$ must have the same self-crossings up to a unit disc automorphism. We prove that an isomorphism between $\M_{V_1}$ and $\M_{V_2}$ can only be given by a composition with a map from $V_1$ to $V_2$. 
    In the case of a single simple self-crossing we show that there are only two possible candidates for this map and find these candidates. Finally, we provide a continuum of $V$'s with the same self-crossing pattern such that their multiplier algebras are all mutually non-isomorphic.
\end{abstract}

\maketitle

\section{Introduction}
    \subsection{The isomorphism problem}
        The main problem we are interested in is the isomorphism problem for multiplier algebras of varieties in the unit ball. In this section we state this problem in the general setting.
    
        Let $\B_d, \ d \in \N \cup \{ \infty \}$ denote the Euclidean open unit ball in $\C^d$, where $\C^d$ means $\ell^2(\N)$ for $d = \infty$. The main function space on $\B_d$ we consider is the Drury-Arveson space.
        \begin{definition}
            \emph{The Drury-Arveson space on $\B_d$} is
            \begin{equation*}
                H^2_d = \left\{ f(z) = \sum_{\a \in \N_0^d} c_{\a} z^{\a}: \: ||f||^2 = \sum_{\a \in \N_0^d} |c_{\a}|^2 \frac{\a!}{|\a|!} < \infty \right\}.
            \end{equation*}
            It is a reproducing kernel Hilbert space with kernel
            \begin{equation*}
                \k(z, w) = \k_w(z) =  \frac{1}{1 - \langle z, w \rangle_{\C^d}}.
            \end{equation*}
        \end{definition}
        This space is one of the multivariate generalizations of the Hardy space $H^2(\D)$ on the unit disc, i.e., $H^2(\D) = H^2_1$.
        The Drury-Arveson space naturally arises in operator theory as well as in theory of complete Pick spaces, see the survey \cite{shalit2015operator} or the more recent \cite{hartz2023invitation} for an introduction.
        
        We denote by $\M_d = \Mult (H^2_d)$ the multiplier algebra of $H^2_d$. We say that $V \subset \B_d$ is \emph{a multiplier variety} if it is a joint zero set of functions from $\M_d$, i.e., there exists $E \subset \M_d$ such that
        \begin{equation*}
            V = \{ z \in \B_d: \f(z) = 0, \text{ for all } \f \in E \}.
        \end{equation*}
        Note that for $d < \infty$ we can replace $E$ by a finite set (even of cardinality at most $d$), see the argument in \cite[Chapter 5, Section 7]{chirka1989complex}. In particular, such $V$ is an analytic variety in $\B_d$. Throughout this paper we simply refer to multiplier varieties as varieties.
        
        For a variety $V \subset \B_d$ the associated Hilbert space of functions on $V$ is $\H_V \subset H^2_d$:
        \begin{equation*}
            \H_V = \Cl \Span \{ \k_{\l}: \l \in V \} = \{ f \in H^2_d: \: f|_V = 0 \}^{\perp}.
        \end{equation*}        
        This Hilbert space is a reproducing kernel Hilbert space of functions on $V$. 
        
        Next, we define the algebra associated to $V$:
        \begin{equation*}
            \M_V = \{ \f|_V: \f \in \M_d \}.
        \end{equation*}
        By \cite[Proposition 2.6]{DSR15}, $\M_V$ is exactly the multiplier algebra $\Mult(\H_V)$, and $\M_V$ is completely isometrically isomorphic to the quotient $\M_d / J_V$, where
        \begin{equation*}
            J_V = \{ \f \in \M_d: \f|_V = 0 \}.
        \end{equation*}
        Moreover, $\M_V$ is contained in $H^{\infty}(V)$, the algebra of bounded analytic functions on $V$, but might not coincide with it.
    
        \textbf{The isomorphism problem:} When is $\M_V \cong \M_W$ equivalent to $V \cong W$? 
        \\
        The answer heavily depends on the notion of isomorphism for both varieties and algebras. 
        We recall a few possible notions. First, $V$ is said to be \emph{an automorphic image} of $W$ in $\B_d$ if there exists $\mu$, an automorphism of $\B_d$, such that $V = \mu(W)$. Since $\mu^{-1}$ is also an automorphism of $\B_d$ this notion is symmetric with respect to $V$ and $W$. We say that $V$ and $W$ are \emph{biholomorphic} if there are holomorphic in $\B_d$ maps $F$ and $G$ such that $F|_V$ is a bijecton from $V$ to $W$ and $G|_W$ is its inverse. Finally, we say that $V$ and $W$ in $\B_d$ are \emph{multiplier biholomorphic} if they are biholomorphic via maps $F, G$ which are coordinate multipliers, i.e.,
        $
            F, G \in \underbrace{\M_d \times \ldots \times \M_d}_{d}.     
        $

        We refer to the survey by Salomon and Shalit \cite{Sal} for an overview of the available answers to the isomorphism problem:
        \begin{itemize}
            \item $\M_V \cong \M_W$ isometrically if and only if $\H_V \cong \H_W$ as reproducing kernel Hilbert spaces if and only if $V$ is an automorphic image of $W$, where $d < \infty$ or $V$ and $W$ have the same affine codimension, \cite[Proposition 4.8, Theorem 4.6]{Sal}, \cite{DSR15}. 
            
            Similarly, Rochberg, \cite[Theorem 7]{rochberg2019complex}, has shown that for ordered finite sets $V, W$, $\H_V \cong \H_W$ if and only if $V$ is an automorphic image of $W$, adding some concrete quantitative conditions for an isomorphism to exist, such as coinciding ``normal forms'' or having the same Gram matrix. Rochberg also showed that $V \cong W$ if and only if all triples in $V$ are congruent to all triples in $W$.
            \item $\M_V \cong \M_W$ algebraically is equivalent to $V$ and $W$ being multiplier biholomorphic for homogeneous varieties $V$ and $W$ if $d < \infty$, \cite[Theorem 5.14]{Sal}, \cite{DRS11}, \cite{HarTop}.
            \item $\M_V \cong \M_W$ algebraically implies that $V$ and $W$ are multiplier biholomorphic, for $V$ and $W$ which are irreducible varieties (or finite unions of irreducible varieties and a discrete variety), \cite[Theorem 5.5]{Sal}, \cite{DSR15}.
            \item $V \cong W$ via a biholomorphism that extends to be a 1-to-1 $C^2$-map on the boundary implies $\M_V \cong \M_W$ algebraically, for $V$ and $W$ --- images of finite Riemann surfaces under a holomap that extends to be a 1-to-1 $C^2$-map on the boundary, \cite[Corollary 5.18]{Sal}. This is the result of an idea from \cite{Vin} being generalized in \cite[Section 2.3.6]{ARS08} and culminating in \cite{Kerr}.
            \item For $V \subset \B_{\infty}$ of the form $V = f(\D)$ with
            \begin{equation*}
                f(z) = (b_1 z, b_2 z^2, b_3 z^3, \ldots), \quad \text{ where } b_1 \ne 0 \text{ and } ||b||_2^2 = 1 
            \end{equation*}
            $\M_V = H^{\infty}(V) \cong H^{\infty}(\D)$ if and only if $\sum n |b_n|^2 < \infty$, while any two such varieties are multiplier biholomorphic, \cite[Corollary 7.4]{Dav}.
        \end{itemize}
         It is still an open question whether $V \cong W$ via a multiplier biholomorphism implies $\M_V \cong \M_W$ algebraically for sufficiently simple $V$ and $W$. Though, it is known to be false for $V$, $W$ --- discrete varieties, see \cite[Example 5.7]{Sal}.

  \subsection{Complete Pick spaces}
        A different way to look at the multiplier algebras $\M_V$ comes from function theory, in particular from the theory of complete Pick spaces. 
        
        Let $\H$ be a reproducing kernel Hilbert space on a set $X$ with kernel $\k$ and let $\M$ be its multiplier algebra $\Mult(\H)$. Consider the following problem. Given $\l_1, \ldots, \l_n \in X$ and $a_1, \ldots, a_n \in \C$, does the interpolation problem
        \begin{equation*}
            \f(\l_j) = a_j, \quad j = 1, \ldots, n
        \end{equation*}
        have a multiplier solution $\f \in \M, \ ||\f||_{\M} \le 1$?
        
        This problem is called \emph{the Pick interpolation problem}\/. Denote by $M_{\f}$ the operator on $\H$ given by multiplying by $\f$. By definition the norm on $\M$ is inherited from the operator norm, hence, $||\f||_{\M} \le 1$ is equivalent to $I_{\H} - M_{\f} M^*_{\f} \ge 0$. Using this fact together with $M^*_{\f} \k_x = \overline{\f(x)} \k_x$ we get a necessary condition for the Pick interpolation problem to have a solution, i.e., 
        \begin{equation} \label{eq: Pick positive matrix}
            \left[ \left( 1 - a_j \overline{a_k} \right) \k(\l_j, \l_k) \right]_{j, k = 1}^n \ge 0.
        \end{equation}
        Pick, \cite{pick1915}, showed that for the Hardy space $H^2(\D)$, which has the Szego kernel 
        \begin{equation*}
            \k(z, w) = \frac{1}{1 - z \bar w},
        \end{equation*}
        this positivity condition is also sufficient. Spaces and their kernels for which the positivity of \eqref{eq: Pick positive matrix} is sufficient for the Pick interpolation problem to have a solution are called \emph{Pick spaces} and \emph{Pick kernels} respectively. It is not difficult to see that not all spaces are Pick, e.g., the Bergman space on the unit disc, namely, the space with the kernel
        \begin{equation*}
            \k(z, w) = \frac{1}{(1 - z \bar w)^2}
        \end{equation*}
        is not Pick.
        
        Let us now turn our attention to the matrix version of the Pick interpolation problem. We consider $\M^{m \times m}$, the algebra of $m \times m$ multiplier matrices, so that $\M^{m \times m} = \Mult(\H^m)$. Then, the $m \times m$ Pick interpolation problem is the following: Given $\l_1, \ldots, \l_n \in X$ and $A_1, \ldots, A_n \in M_m(\C)$, does the interpolation problem
        \begin{equation*}
            F(\l_j) = A_j, \quad j = 1, \ldots, n
        \end{equation*}
        have an $m \times m$ multiplier solution $F \in \M^{m \times m}, \ ||F||_{\M^{m \times m}} \le 1$?
        
        Similarly to the $m = 1$ case, it is possible to show the necessity of the positivity condition
        \begin{equation*}
            \left[ \left( 1 - A_j A_k^* \right) \k(\l_j, \l_k) \right]_{j, k = 1}^n \ge 0.
        \end{equation*}
        If this condition is sufficient for all $m \ge 1$ we call the space and the respective kernel \emph{complete Pick}.  
        We refer to the monograph of Agler and McCarthy, \cite{AglBook}, for an in-depth introduction to the study of (complete) Pick kernels, and to \cite[Chapter 5, Section 3]{quiggin1994generalisations} for an example of a space which is Pick but not complete Pick. 

        We say that an RKHS $\H$ on $X$ with kernel $\k$ is \emph{irreducible} if $\k(x, y)$ does not vanish for $x, y \in X$.
        The connection between complete Pick spaces and our isomorphism problem is given by the following theorem due to Agler and McCarthy, \cite[Theorem 4.2]{Agl}, \cite[Theorem 8.2]{AglBook}. 
        \begin{thrmm} \label{thrmm: Agler-McCarthy}
            Let $\H$ be a separable irreducible reproducing kernel Hilbert space on $X$. Then $\H$ is complete Pick if and only if for some $d \in \N \cup \{ \infty \}$ there is a map $b: X \to \B_d$ and a nowhere vanishing function $\de: X \to \C$ such that
            \begin{equation*}
                \k(z, w) = \frac{\de(z) \overline{\de(w)}}{1 - \langle b(z), b(w) \rangle_{\C^d}}.
            \end{equation*}
            Meaning that up to rescaling and composition $\H$ is essentially $\H_V$ for some variety $V \subset \B_d$. Moreover, $\M$ is isometrically isomorphic to $\M_V$ via the same composition.
        \end{thrmm}
        
        We see that the Drury-Arveson space is, in a sense, the universal complete Pick space. Hence, the isomorphism problem is closely related to the study of irreducible complete Pick spaces and their multiplier algebras.

    \subsection{Analytic Discs} \label{subsec: Analytic Discs}   
    
        In this paper we study the isomorphism problem in the case where $V$ and $W$ are analytic discs attached to the unit sphere. 
        \begin{definition} \label{def: analytic disc}
            \emph{An analytic disc attached to the unit sphere} is a variety $V \subset \B_d, \, d < \infty$ for which there exists an injective analytic map $f: \D \to \B_d$ with $f'(z) \ne 0, \, z \in \D$ such that $V = f(\D)$, $f$ extends to $C^2$ up to $\overline{\D}$ and 
            \begin{equation*}
                ||f(x)|| = 1 \iff |x| = 1.  
            \end{equation*}            
            We say that $f$ is \emph{an embedding map} of $V$.
        \end{definition}
        Note that by \cite[Corollary 3.2]{Dav} such discs meet the boundary transversally, i.e., 
        \begin{equation*}
            \langle f(\xi), f'(\xi) \rangle \ne 0, \quad \xi \in \TT.     
        \end{equation*}
        We emphasize that $V$ is defined to be a variety. It is not clear whether for an arbitrary $f$ satisfying properties from Definition \ref{def: analytic disc} the image $V = f(\D)$ is a variety.
        
        Instead of working directly with the spaces $\H_V$ and $\M_V$, it is more convenient to pull back this spaces from the variety $V$ to the unit disc $\D$.
        We denote by $\H_f$ an RKHS on $\D$ that we get from $\H_V$ by composing with $f$:
        \begin{equation*}
            \H_f = \{ h:\D \to \C: \: h \circ f^{-1} \in \H_V \}, \qquad ||h||_{\H_f} = || h \circ f^{-1} ||_{\H_V}.
        \end{equation*}
        The reproducing kernel of this space is
        \begin{equation} \label{eq: k^f formula}
            \k^f(z, w) = \frac{1}{1 - \langle f(z), f(w) \rangle}.
        \end{equation}
        We denote the multiplier algebra of this space by $\M_f$, so that
        \begin{equation*}
            \M_f = \{ \f:\D \to \C: \: \f \circ f^{-1} \in \M_V \}, \qquad ||\f||_{\M_f} = || \f \circ f^{-1} ||_{\M_V}.
        \end{equation*}
    
        The isomorphism problem was partially solved in the case when $f$ extends injectively to $\overline{\D}$. We state the results \cite[Proposition 2.2, Theorem 2.3]{Vin} combined with the above-mentioned fact that the transversality condition is satisfied automatically.
        \begin{thrmm} \label{thrmm: injective embedding}
            Suppose $V$ is an analytic disc attached to the unit sphere, $f$ is the embedding map of $V$. If the extension of $f$ to $\overline{\D}$ is injective, then 
            \begin{itemize}
                \item $\H_f = H^2(\D)$ with equivalent norms,
                \item $\M_f = H^{\infty}(\D)$ with equivalent norms.
            \end{itemize}
        \end{thrmm}
    
        We now state our results.
        \begin{thrm} \label{thrm: algebras to crossings}
            Suppose $V$, $W$ are analytic discs attached to the unit sphere, and $f$, $g$ are the respective embedding maps. If $\M_f = \M_g$, then
            \begin{equation*}
                f(\xi) = f(\zeta) \iff g(\xi) = g(\zeta), \quad \xi, \zeta \in \TT. 
            \end{equation*}
            This means that $f$ and $g$ have the same self-crossings on the boundary.
        \end{thrm}
        \begin{thrm} \label{thrm: isomorphic algebras to automorphism}
            Suppose $V$, $W$ are analytic discs attached to the unit sphere, and $f$, $g$ are the respective embedding maps. If $\M_f \cong \M_g$ algebraically, then there exists $\mu$, an automorphism of the unit disc $\D$, such that $\M_f = \M_{g \circ \mu}$ with equivalent norms.
        \end{thrm}
        Combining these two theorems we get: 
        \begin{crl} \label{thrm: algebras to self-crossings}
            Suppose $V$, $W$ are analytic discs attached to the unit sphere, and $f$, $g$ are the respective embedding maps. If $\M_f \cong \M_g$ algebraically, then, up to a unit disc automorphism, $f$ and $g$ have the same self-crossings on the boundary.
        \end{crl}
        If $f$ and $g$ have the same self-crossings on the boundary up to a unit disc automorphism we say that they have the same \emph{self-crossing type}. Corollary \ref{thrm: algebras to self-crossings} means that for $f$ and $g$ with different self-crossing types the multiplier algebras are always non-isomorphic.
        
        Note that now we have a more complete solution in the injective case:
        Corollary \ref{thrm: algebras to self-crossings} together with Theorem \ref{thrmm: injective embedding} give us
        \begin{crl}
            Suppose $V$, $W$ are analytic discs attached to the unit sphere, and $f$, $g$ are the respective embedding maps and $f$ is injective up to $\overline{\D}$. Then $\M_f \cong \M_g$ algebraically if and only if $g$ is injective up to $\overline{\D}$.  
        \end{crl}
        
        In general, in light of Corollary \ref{thrm: algebras to self-crossings}, we found a coarse characteristic which separates non-isomorphic cases, so that it remains to solve the isomorphism problem only for embeddings with the same self-crossing type.
        However, even in the case of embeddings with the same self-crossing type there might be more than one isomorphism class of multiplier algebras, which we will show in the next section.

    \subsection{First order equivalence and examples} \label{subsec: A and examples}
        The results from the previous section can be thought of as the zeroth order equivalence, i.e. the alignments of values of the embedding map $f$. In this section we consider the effect that the derivative of $f$ has on the isomorphism problem. 
        
        Define $A_f(\xi) = \langle f(\xi), f'(\xi) \xi \rangle, \ \xi \in \TT$. From \cite[Corollary 3.2]{Dav} we know that $A_f(\xi) > 0$ for embedding maps $f$. We prove a theorem that shows that this function on the circle is an important semi-invariant for $\M_f$:
        \begin{thrm} \label{thrm: multipliers to bilip}
            Suppose $V$, $W$ are analytic discs attached to the unit sphere, and $f$, $g$ are the respective embedding maps. If $\M_f = \M_g$ with equivalent norms, then for $\xi, \z \in \TT$ such that $f(\xi) = f(\z)$ (and so Theorem \ref{thrm: algebras to crossings} implies $g(\xi) = g(\z)$) we have
            \begin{equation*}
                \frac{A_f(\xi)}{A_f(\z)} = \frac{A_g(\xi)}{A_g(\z)}.
            \end{equation*}
        \end{thrm}
        If there is \emph{an $n$-point self-crossing}, i.e., $f(\xi_1) = \ldots = f(\xi_n)$, we can apply this theorem pairwise to obtain:
        \begin{crl} \label{crl: n point A A A}
            If there is an $n$-point self-crossing $f(\xi_1) = \ldots = f(\xi_n)$, then
            \begin{equation*}
                [A_f(\xi_1):\ldots:A_f(\xi_n)] = [A_g(\xi_1):\ldots:A_g(\xi_n)],
            \end{equation*}
            by which we mean that there is a scaling factor $c > 0$ such that $A_f(\xi_j) = c A_g(\xi_j), \ j = 1, \ldots, n$. 
        \end{crl}
        This theorem allows us to better understand some sub-classes of the isomorphism problem with specific self-crossing types. 

        First, we consider analytic discs with exactly one simple self-crossing on the boundary. Since we can always apply an automorphism, we assume that $f(-1) = f(1)$ is the self-crossing and otherwise $f$ is injective in $\overline{\D}$. 
        We use Theorem \ref{thrm: multipliers to bilip} to show that in this case the isomorphism condition is rigid.
        \begin{thrm} \label{thrm: multiplier isomorphism two points to two auto}
            Suppose $V$, $W$ are analytic discs attached to the unit sphere, and $f$, $g$ are the respective embedding maps with the only self-crossing at $\pm 1$. Then $\M_f \cong \M_g$ algebraically if and only if $\M_f = \M_{g \circ \mu}$, where either
            \begin{equation*}
                \mu(z) = \frac{z - \a}{1 - \a z}, \text{ or} \qquad \mu(z) = \frac{\b - z}{1 - \b z}
            \end{equation*}
            with constants $\a, \b \in (-1, 1)$ that can be explicitly computed in terms of $A_f(\pm 1)$ and $A_g(\pm 1)$:
            \begin{equation*}
                \a = \frac{\sqrt{A_f(1)A_g(-1)} - \sqrt{A_f(-1)A_g(1)}}{\sqrt{A_f(1)A_g(-1)} + \sqrt{A_f(-1)A_g(1)}}, \qquad \b = \frac{\sqrt{A_f(1)A_g(1)} - \sqrt{A_f(-1)A_g(-1)}}{\sqrt{A_f(1)A_g(1)} + \sqrt{A_f(-1)A_g(-1)}}.
            \end{equation*}
        \end{thrm}   

        Let us consider an example. Set
        \begin{equation*}
            b_r(z) = \frac{z - r}{1 - rz}, \quad -1 < r < 1,
        \end{equation*}
        exactly the automorphisms of the unit disc such that $b_r(\pm 1) = \pm 1$. Define
        \begin{equation*}
            f_r(z) = \frac{1}{\sqrt{2}} \left( z^2, b_r(z)^2 \right), \quad r > 0.
        \end{equation*}
        Then, according to \cite[Theorem 5.2]{Dav}, $V = f(\D)$ is a variety with the only self-crossing at $\pm 1$. Using Theorem \ref{thrm: multiplier isomorphism two points to two auto} we show:
        \begin{thrm} \label{thrm: f_r and f_s}
            $\M_{f_r} \ne \M_{f_s}$ for $r \ne s$.
        \end{thrm}
        We see that, in contrast to the injective case, even for embeddings with the same self-crossings the multiplier algebras might not coincide. However, it does not mean that the algebras are not isomorphic, as we still have the freedom to do unit disc automorphisms. Let us take care of this by considering a more general version of this embedding. Set
        \begin{equation*}
            f_{r, s}(z) = \frac{1}{\sqrt{2}} \left( b_r(z)^2, b_s(z)^2 \right), \quad r \ne s,
        \end{equation*}
        so that $f_r = f_{0, r}$.
        Note that
        \begin{equation*}
            (b_r \circ b_s)(z) = \frac{z - \frac{r + s}{1 + rs}}{1 - \frac{r + s}{1 + rs} z} = b_{\frac{r + s}{1 + rs}}(z).
        \end{equation*}
        Hence,
        \begin{equation*}
            (f_{r, s} \circ b_{-r})(z) = \frac{1}{\sqrt{2}} \left( z^2, b_{\frac{s - r}{1 - sr}}(z)^2 \right) = f_{0, \frac{s - r}{1 - sr}}(z),
        \end{equation*}
        meaning that we get the same varieties (do $z \mapsto -z$ in case $s < r$). Now, for $f_{0, r}, \, r \in (0,1)$ consider 
        \begin{equation*}
            t = \frac{-1 + \sqrt{1 - r^2}}{r},
        \end{equation*}
        we have $t \in (-1, 0)$ and $f_{0, r} \circ b_t = f_{t, -t}$. In fact, $t$ is unique, so that any $f_{r, s}$ corresponds to a unique $f_{t, -t}, \, t \in (-1, 0)$ via a unit disc automorphism. 
        For such embeddings Theorem \ref{thrm: multiplier isomorphism two points to two auto} takes a simpler form
        \begin{thrm} \label{thrm: f_r-r and f_s-s}
            $\M_{f_{r, -r}} \cong \M_{f_{s, -s}}$ algebraically if and only if  $\M_{f_{r, -r}} = \M_{f_{s, -s}}$.
        \end{thrm}
        The author does not know whether $\M_{f_{r, -r}} = \M_{f_{s, -s}}$ holds for any $r$ and $s$, but it seems plausible. If there are $r$ and $s$ for which $\M_{f_{r, -r}} \ne \M_{f_{s, -s}}$, then Theorem \ref{thrm: f_r-r and f_s-s} gives us an example where two embeddings with a single simple self-crossing give rise to non-isomorphic multiplier algebras. This means that whether there is only one isomorphism class of multiplier algebras remains unclear in the single simple self-crossing case.

        We finish by proving that already for the single $3$-point self-crossing case the converse to the Corollary \ref{thrm: algebras to self-crossings} does not hold. There are, in fact, quite a few isomorphism classes of multiplier algebras.
        \begin{thrm} \label{thrm: example single 3-self-crossing}
            There is a continuum of analytic discs attached to the unit sphere with the same $3$-point self-crossing and such that their multiplier algebras are all mutually non-isomorphic.  
        \end{thrm}
        For the example appearing in this theorem we rely on Theorem \ref{thrm: multipliers to bilip} to prove that two multiplier algebras are non-isomorphic. It, hence, remains an open question whether the converse to Theorem \ref{thrm: multipliers to bilip} is true: If the embedding maps have the same self-crossing type and an equivalent derivative behaviour does that imply that the multiplier algebras are isomorphic? 
    \subsection{Organization of the paper}
        In Section \ref{sec: Analytic Discs} we prove the results stated in Subsection \ref{subsec: Analytic Discs}: First, we study the spaces $\H_f$ and $\M_f$ in general in Subsection \ref{subsec: Hf and Mf}. Then, we prove Theorems \ref{thrm: algebras to crossings} and \ref{thrm: isomorphic algebras to automorphism} in Subsections \ref{subsec: Theorem 1} and \ref{subsec: Theorem 2} respectively.

        Section \ref{sec: A and examples} is entirely devoted to the proofs of the results stated in Subsection \ref{subsec: A and examples}. We prove Theorems \ref{thrm: multipliers to bilip}, \ref{thrm: multiplier isomorphism two points to two auto}, \ref{thrm: f_r and f_s}, \ref{thrm: f_r-r and f_s-s} and \ref{thrm: example single 3-self-crossing} in Subsections \ref{subsec: Theorem 3}, \ref{subsec: Theorem 4}, \ref{subsec: Theorem 5}, \ref{subsec: Theorem 6} and \ref{subsec: Theorem 7} respectively.
        
\section{Analytic Discs} \label{sec: Analytic Discs}
    \subsection{Properties of functions in \texorpdfstring{$\H_f$}{Hf} and \texorpdfstring{$\M_f$}{Mf}} \label{subsec: Hf and Mf}
    We start by studying some necessary properties which must be satisfied by functions in the Hilbert spaces $\H_f$ and the multiplier algebras $\M_f$.
    \begin{thrm} \label{thrm: 1 = -1}
        Suppose $V$ is an analytic disc attached to the unit sphere, and $f$ is its embedding map such that $f(-1) = f(1)$. Let $h$ be a function in $\H_f$ such that the following limits exist
        \begin{equation*}
            h(1) = \lim_{r \to 1-} h(r), \qquad h(-1) = \lim_{r \to 1-} h(-r).
        \end{equation*}
        Then $h(1) = h(-1)$.
    \end{thrm}
    In particular, if $h \in \H_f$ extends continuously up to $\overline \D$, then $h(1) = h(-1)$.

    \begin{proof}
        By appropriately scaling $h$ we can assume $||h||_{\H_f} \le 1$. Since $\H_f$ is an RKHS with kernel $\k^f$ we can use \cite[Theorem 3.11]{paulsen2016introduction} to describe the functions in $\H_f$ in terms of the kernel:
        \begin{equation*}
            ||h||_{\H_f} \le 1 \iff \left( \k^f(z,w) - h(z) \overline{h(w)} \right)_{z, w \in \D} \ge 0.       
        \end{equation*}  
        In particular, using \eqref{eq: k^f formula}, for two points $z, w \in \D$ we obtain
        \begin{equation*}
            \begin{pmatrix}
                \frac{1}{1 - || f(z) ||^2} - |h(z)|^2 & \frac{1}{1 - \langle f(z), f(w) \rangle} - h(z) \overline{h(w)} \\
                \frac{1}{1 - \langle f(w), f(z) \rangle} - h(w) \overline{h(z)} & \frac{1}{1 - || f(w) ||^2} - |h(w)|^2
            \end{pmatrix}
            \ge 0.
        \end{equation*}
        As this condition implies that the determinant is positive, we have:
        \begin{align} 
            \left( \frac{1}{1 - || f(z) ||^2} - |h(z)|^2 \right) \left( \frac{1}{1 - || f(w) ||^2} - |h(w)|^2 \right) - \nonumber \\
            \label{eq: det>0}
            \left| \frac{1}{1 - \langle f(z), f(w) \rangle} - h(z) \overline{h(w)} \right|^2 \ge 0.
        \end{align}
    Now let us take $z = 1 - x, w = -1 + y$ for $x, y > 0$. Expanding $f$ around $1$ and $-1$ we get
    \begin{align*}
        f(1 - x) &= f(1) - f'(1)x + \frac{f''(1)}{2} x^2 + o(x^2), \\ 
        f(-1 + y) &= f(-1) + f'(-1)y + \frac{f''(-1)}{2} y^2 + o(y^2).
    \end{align*}
    Thus, 
    \begin{align} 
        & 1 - ||f(1 - x)||^2 = 
        2 \Re \langle f(1), f'(1) \rangle x - (||f'(1)||^2 + \Re \langle f(1), f''(1) \rangle)x^2 + o(x^2). \label{eq: 11.1}
    \end{align}
    Similarly,
    \begin{align} 
        & 1 - ||f(-1 + y)||^2 = 
        -2 \Re \langle f(-1), f'(-1) \rangle y -  (||f'(-1)||^2 + \Re \langle f(-1), f''(-1) \rangle)y^2 + o(y^2). \label{eq: -1-1.1}
    \end{align}
    Note that by \cite[Proposition 3.1, Corollary 3.2]{Dav}, $A = \langle f(1), f'(1) \rangle > 0$ and $B = -\langle f(-1), f'(-1) \rangle > 0$. Let us set
    \begin{align*}
        C &= ||f'(1)||^2, \\
        2F &= \langle f(1), f''(1) \rangle, \\
        D &= ||f'(-1)||^2, \\
        2G &= \langle f(-1), f''(-1) \rangle.   
    \end{align*}
    We get
    \begin{align}
        1 - ||f(1 - x)||^2 &=  2Ax - (C + 2 \Re F)x^2 + o(x^2), \label{eq: 11.2} \\
        1 - ||f(-1 + y)||^2 &=  2By - (D + 2 \Re G)y^2 + o(y^2). \label{eq: -1-1.2}
    \end{align}
    Finally, noting that $f(1) = f(-1)$ and setting $E = \langle f'(1), f'(-1) \rangle$, we get
    \begin{equation} \label{eq: 1-1} 
        1 - \langle f(1-x), f(-1 + y) \rangle = Ax + By + E xy - \bar Fx^2 - Gy^2 + o(x^2) + o(y^2).  
    \end{equation}
    Expanding the brackets in \eqref{eq: det>0} we get
    \begin{align}
        \frac{1}{1 - || f(z) ||^2} \frac{1}{1 - || f(w) ||^2} - \left| \frac{1}{1 - \langle f(z), f(w) \rangle} \right|^2 \ge \label{eq: det:LHS} \\
        \frac{|h(w)|^2}{1 - || f(z) ||^2} + \frac{|h(z)|^2}{1 - || f(w) ||^2} - 2 \Re \left( \frac{\overline{h(z)} h(w)}{1 - \langle f(z), f(w) \rangle} \right) \label{eq: det:RHS}
    \end{align}
    To annihilate the highest order term in \eqref{eq: det:LHS} we set $Ax = By = t > 0$. Hence,
    \begin{align}
        \frac{1}{1 - || f(z) ||^2} = 
        \frac{1}{2t - \frac{C + 2 \Re F}{A^2}t^2 + o(t^2)} = \nonumber \\
        \frac{1}{2t} \left(1 + \frac{C + 2 \Re F}{2A^2}t + o(t) \right). \label{eq: 11.3}
    \end{align}
    Similarly,
    \begin{align}
        \frac{1}{1 - || f(w) ||^2} = 
        \frac{1}{2t - \frac{D + 2 \Re G}{B^2}t^2 + o(t^2)} = \nonumber \\
        \frac{1}{2t} \left( 1 + \frac{D + 2 \Re G}{2B^2}t + o(t) \right). \label{eq: -1-1.3}
    \end{align}
    Finally,
    \begin{align} 
        \frac{1}{1 - \langle f(z), f(w) \rangle} = 
        \frac{1}{2t + \left( \frac{E}{AB} - \frac{\bar F}{A^2} - \frac{G}{B^2} \right) t^2 + o(t^2)} = \nonumber \\
        \frac{1}{2t} \left(1 + \left( \frac{\bar F}{2A^2} + \frac{G}{2B^2} - \frac{E}{2AB} \right) t + o(t) \right). \label{eq: 1-1.2} 
    \end{align}
    Hence,
    \begin{align} 
        \left| \frac{1}{1 - \langle f(z), f(w) \rangle} \right|^2 = 
        \frac{1}{4t^2} \left( 1 + \left( \frac{\Re F}{A^2} + \frac{\Re G}{B^2} - \frac{\Re E}{AB} \right) t + o(t) \right). \label{eq: 1-1.3}   
    \end{align}
    Substituting \eqref{eq: 11.3}, \eqref{eq: -1-1.3} and \eqref{eq: 1-1.3} into \eqref{eq: det:LHS} we get
    \begin{align}
        \frac{1}{4t^2} \Biggl( \left(1 + \frac{C + 2 \Re F}{2A^2}t + o(t) \right) \left( 1 + \frac{D + 2 \Re G}{2B^2}t + o(t) \right) - \nonumber \\
        \left( 1 + \left( \frac{\Re F}{A^2} + \frac{\Re G}{B^2} - \frac{\Re E}{AB} \right) t + o(t) \right) \Biggr) = \nonumber \\  
        \frac{1}{4t} \left( \frac{C}{2A^2} + \frac{\Re F}{A^2} + \frac{D}{2B^2} + \frac{\Re G}{B^2} - \frac{\Re F}{A^2} - \frac{\Re G}{B^2} + \frac{\Re E}{AB} + o(1) \right) = \nonumber \\
        \frac{1}{8t} \left( \frac{C}{A^2} + \frac{D}{B^2} + \frac{2 \Re E}{AB} \right) + o \left( \frac 1 t \right). \label{eq: det:LHS.2}
    \end{align}
    Similarly, expanding only up to $o(1)$ in \eqref{eq: 11.3}, \eqref{eq: -1-1.3}, \eqref{eq: 1-1.2} and substituting into \eqref{eq: det:RHS} we get 
    \begin{align}
        \frac{1}{2t} |h(w)|^2 (1 + o(1)) + \frac{1}{2t} |h(z)|^2 (1 + o(1)) - \frac{1}{2t} 2 \Re (\overline{h(z)} h(w)) (1 + o(1)) = \nonumber \\
        \frac{1}{2t} \left( |h(z)|^2 + |h(w)|^2 - 2 \Re (\overline{h(z)} h(w)) \right) + o \left( \frac 1 t \right) = 
        \nonumber \\ 
        \frac{1}{2t} |h(z) - h(w)|^2 + o \left( \frac 1 t \right). \label{eq: det:RHS.2}
    \end{align}
    Thus, replacing \eqref{eq: det:LHS} with \eqref{eq: det:LHS.2} and \eqref{eq: det:RHS} with \eqref{eq: det:RHS.2} we get
    \begin{equation*}
        \frac{1}{8t} \left( \frac{C}{A^2} + \frac{D}{B^2} + \frac{2 \Re E}{AB} \right) + o \left( \frac 1 t \right) \ge \frac{1}{2t} |h(z) - h(w)|^2 + o \left( \frac 1 t \right).  
    \end{equation*}
    This is equivalent to
    \begin{equation*}
        |h(z) - h(w)|^2 \le \frac{1}{4} \left( \frac{C}{A^2} + \frac{D}{B^2} + \frac{2 \Re E}{AB} \right) + o(1).    
    \end{equation*}
    Explicitly writing $z = 1 - t/A$ and $w = - 1 + t/B$ we get 
    \begin{equation} \label{eq: k - k}
         \left|h \left( 1 - \frac{t}{A} \right) - h \left( -1 + \frac{t}{B} \right) \right|^2 \le \frac{1}{4} \left( \frac{C}{A^2} + \frac{D}{B^2} + \frac{2 \Re E}{AB} \right) + o(1), \quad t > 0.
    \end{equation}
    In particular, taking the limit as $t \to 0$ for $h$ that have radial limits at $\pm 1$ with $|| h ||_{\H_f} \le 1$ we get
    \begin{equation*}
        \left|h (1) - h \left( -1  \right) \right|^2 \le \frac{1}{4} \left( \frac{C}{A^2} + \frac{D}{B^2} + \frac{2 \Re E}{AB} \right).
    \end{equation*}
    It follows by scaling that in general we have
    \begin{equation*}
        \left|h (1) - h \left( -1  \right) \right|^2 \le \frac{1}{4} \left( \frac{C}{A^2} + \frac{D}{B^2} + \frac{2 \Re E}{AB} \right) || h ||^2_{\H_f},
    \end{equation*}
    for all $h \in \H_f$ with radial limits at $\pm 1$. Note that such functions are dense in $\H_f$, since, for example, all $\k^f_z, \, z \in \D$, \eqref{eq: k^f formula}, are continuous in $\overline \D$. 

    We conclude that there is a bounded functional on $\H_f$, let us denote the respective function by $\f \in \H_f$, such that $\langle h, \f \rangle_{\H_f} = h(1) - h(-1)$ for $h \in \H_f$ that have the limits 
    \begin{equation*}
        h(1) = \lim_{x \to 0+} h(1 - x), \qquad h(-1) = \lim_{y \to 0+} h(-1 + y).
    \end{equation*}
    To finish the proof of the theorem it remains to show that $\f = 0$. Indeed, for any $z \in \D$
    \begin{equation*}
        \f (z) = \langle \f, \k^f_z \rangle_{\H_f} = \overline{\langle \k^f_z, \f \rangle_{\H_f}} = \overline{\k^f_z(1) - \k^f_z(-1)} = 0.
    \end{equation*}
    \end{proof}

    \begin{crl} \label{crl: k - k w 0}
        In the setting of Theorem \ref{thrm: 1 = -1} we have
        \begin{equation*}
            \k^f_{1 - \frac t A} - \k^f_{-1 + \frac t B} \to 0, \quad t \to 0.
        \end{equation*}
        in the $w^*$ topology.
    \end{crl}
    \begin{proof}
        From \eqref{eq: k - k} it follows that 
        \begin{equation*}
            \k^f_{1 - \frac t A} - \k^f_{-1 + \frac t B}, \quad t \to 0.
        \end{equation*}
        is bounded in norm. Hence, since linear combinations of $\k^f_z, \ z \in \D$ are dense in $\H_f$, the conclusion follows from Theorem \ref{thrm: 1 = -1}. 
    \end{proof}

    An immediate corollary of Theorem \ref{thrm: 1 = -1} is the following
    \begin{thrm} \label{crl: xi = zet}
        Suppose $V$ is an analytic disc attached to the unit sphere, and $f$ is its embedding map. Let $f(\xi) = f(\z)$ for two distinct $\xi, \z \in \mathbb{T}$. If $h$ is a function in $\H_f$ that is continuous up to $\overline \D$, then $h(\xi) = h(\z)$.
    \end{thrm}
    \begin{proof}
        Consider a disc automorphism $\mu$ such that $\mu(-1) = \xi, \, \mu(1) = \z$. Then $g = f \circ \mu$ satisfies the conditions of Theorem \ref{thrm: 1 = -1}. Note that $h \circ \mu$ belongs to $\H_g$ ($\H_f \cong \H_g$ as RKHS by $\circ \mu$). $h \circ \mu$ is continuous up to $\overline \D$ as well, so it satisfies the conditions of Theorem \ref{thrm: 1 = -1}. We conclude $(h \circ \mu)(-1) = (h \circ \mu) (1)$, which is exactly $h(\xi) = h(\z)$.
    \end{proof} 

    From this result we can infer that the same must hold for multipliers.
    \begin{thrm} \label{thrm: crossings to multipliers}
        Suppose $V$ is an analytic disc attached to the unit sphere, and $f$ is its embedding map. Let $f(\xi) = f(\z)$ for two distinct $\xi, \z \in \mathbb{T}$. If $\f$ is a function in $\M_f$ that is continuous up to $\overline \D$, then $\f(\xi) = \f(\z)$.
    \end{thrm}
    \begin{proof}
        Consider $h(z) = \k_{0}^f(z)$. By definition $h(\xi) = h(\z) \ne 0$. Note that $\f h \in \H_f$ is continuous up to $\overline{\D}$, so that $\f(\xi) h(\xi) = \f(\z) h(\z)$. Dividing by $h(\xi) = h(\z) \ne 0$ we get $\f(\xi) = \f(\z)$.
    \end{proof}

    \subsection{Proof of Theorem \ref{thrm: algebras to crossings}} \label{subsec: Theorem 1}
    \begin{thrmrestate}
            Suppose $V$, $W$ are analytic discs attached to the unit sphere, and $f$, $g$ are the respective embedding maps. If $\M_f = \M_g$, then
            \begin{equation*}
                f(\xi) = f(\zeta) \iff g(\xi) = g(\zeta), \quad \xi, \zeta \in \TT. 
            \end{equation*}
            This means that $f$ and $g$ have the same self-crossings on the boundary.        
    \end{thrmrestate}

    \begin{proof}
        The result follows from Theorem \ref{thrm: crossings to multipliers}. Let us write $f$ as $f = (f_1, \ldots, f_d)$. Note that $z_j \in \M_d, \ j = 1, \ldots, d$, so that if we compose with $f$ we get $f_j \in \M_f, \ j = 1, \ldots, d$. Since $\M_f = \M_g$, by Theorem \ref{thrm: crossings to multipliers}, $g(\xi) = g(\z) \implies f_j(\xi) = f_j(\z), \ j = 1, \ldots, d$, which means exactly $f(\xi) = f(\z)$. Exchanging $f$, $g$ and repeating this argument we get the desired equivalence.    
    \end{proof}

    \subsection{Proof of Theorem \ref{thrm: isomorphic algebras to automorphism}} \label{subsec: Theorem 2}
    First, we show that analytic discs are irreducible.
    \begin{lem} \label{lem: irreducible varieties}
        If $V$ is an analytic disc attached to the unit sphere, then $V$ is irreducible in the sense of \cite[Section 5.1]{Sal}, i.e., for any regular point $\l \in V$ and any $\e > 0$ the intersection of zero sets of all multipliers vanishing on $V \cap B_{\e}(\l)$, a small neighborhood of $\l$ in $V$, is exactly $V$.
    \end{lem}
    \begin{proof}
        Suppose $V$ is not irreducible. This means that there is $\l \in V$ and $\e > 0$ such that the intersection of zero sets of all multipliers vanishing on $V \cap B_{\e}(\l)$ is a proper subset of $V$. Hence, there is $v \in V$ and $\f \in \M_d$ such that $\f|_{V \cap B_{\e}(\l)} = 0$ but $\f(v) \ne 0$. Let $f$ be the embedding map of $V$. Then, $v = f(a)$ for some $a \in \D$. Consider a function $\psi = \f \circ f$, it is analytic in $\D$. Since $\f|_{V \cap B_{\e}(\l)} = 0$, then $\psi|_{f^{-1}(B_{\e}(\l))} = 0$. But the set $f^{-1}(B_{\e}(\l))$ is open in $\D$, which means that $\psi = 0$. This contradicts the fact that $\psi(a) = \f(v) \ne 0$.
    \end{proof}
    
    Now, we can prove Theorem \ref{thrm: isomorphic algebras to automorphism}:
    \begin{thrmrestate}
        Suppose $V$, $W$ are analytic discs attached to the unit sphere, and $f$, $g$ are the respective embedding maps. If $\M_f \cong \M_g$ algebraically, then there exists $\mu$, an automorphism of the unit disc $\D$, such that $\M_f = \M_{g \circ \mu}$ with equivalent norms.
    \end{thrmrestate}
    \begin{proof}
        Suppose that $\M_f \cong \M_g$ algebraically. 
        This is equivalent to $\M_V \cong \M_W$. Denote by $\Phi: \M_V \to \M_W$ the isomorphism.
        By Lemma \ref{lem: irreducible varieties}, the varieties $V$ and $W$ are irreducible. 
        Thus, by \cite[Theorem 5.5]{Sal}, there are maps $F, G: \B_d \to \B_d$ with multiplier coefficients, i.e., $F_j, G_j \in \M_d, \ j = 1, \ldots, d$, such that $F \circ G = \id_V$, $G \circ F = \id_W$ and the composition with these maps gives us $\Phi$, i.e.,
        \begin{equation*}
            \Phi(\f) = \f \circ F, \ \f \in \M_V, \qquad \Phi^{-1}(\psi) = \psi \circ G, \ \psi \in \M_W. 
        \end{equation*}
        To go back to the unit disc we define
        \begin{equation*}
            \mu = g^{-1} \circ G \circ f : \D \to \D,
        \end{equation*}
        so that
        \begin{equation*}
            \mu^{-1} = f^{-1} \circ F \circ g: \D \to \D.
        \end{equation*}
        Since $f^{-1}, g^{-1}$ are analytic on $V$, $W$ respectively, we see that $\mu$ is an analytic bijection from $\D$ onto itself, and, hence, a disc automorphism.
        
        By definition,
        \begin{equation*}
            \f \in \M_{g \circ \mu} \iff \f \circ \mu^{-1} \circ g^{-1} \in \M_W \iff \f \circ \mu^{-1} \circ g^{-1} \circ G \in \M_V \iff
        \end{equation*}
        \begin{equation*}
            \f \circ \mu^{-1} \circ g^{-1} \circ G \circ f \in \M_f \iff \f \in \M_f.
        \end{equation*}
        Since $\Phi$ is continuous as a homomorphism of semi-simple Banach algebras and the norms on $\M_f, \M_g$ are inherited from $\M_V, \M_W$ we conclude that $\M_f = \M_{g \circ \mu}$ with equivalent norms.    
    \end{proof}

\section{First order equivalence and examples} \label{sec: A and examples}
    We begin this section by considering how the semi-invariant $A_f(\xi) = \langle f(\xi), f'(\xi) \xi \rangle, \ \xi \in \TT$ changes under unit disc automorphisms.
        \begin{lem} \label{lem: A_f under mu}
            Suppose $f$ is an embedding map, $\mu$ is a disc automorphism
            \begin{equation*}
                \mu(z) = \l \frac{\a - z}{1 - \bar \a z}, \quad \l \in \TT, \, \a \in \D.
            \end{equation*}
            Then
            \begin{equation*}
                A_{f \circ \mu}(\xi) = A_f (\mu(\xi)) \frac{1 - |\a|^2}{|\a - \xi|^2}. 
            \end{equation*}
        \end{lem}
        \begin{proof}
            Expand
            \begin{align*}
                A_{f \circ \mu}(\xi) = \langle f(\mu(\xi)), (f \circ \mu)'(\xi) \xi \rangle = \langle f(\mu(\xi)), f'(\mu(\xi)) \mu'(\xi) \xi \rangle &= \\
                \langle f(\mu(\xi)), f'(\mu(\xi)) \mu(\xi) \frac{\mu'(\xi)}{\mu(\xi)} \xi \rangle &= \\
                \overline{\frac{\mu'(\xi)}{\mu(\xi)} \xi} A_f (\mu(\xi)).
            \end{align*}
            Hence, it is enough to show
            \begin{equation*}
                \frac{\mu'(\xi)}{\mu(\xi)} \xi = \frac{1 - |\a|^2}{|\a - \xi|^2}.  
            \end{equation*}
            Indeed,
            \begin{equation*}
                \mu'(z) = \l \frac{|\a|^2 - 1}{(1 - \bar \a z)^2},
            \end{equation*}
            so that
            \begin{equation*}
                \frac{\mu'(z)}{\mu(z)} = \frac{|\a|^2 - 1}{(\a - z)(1 - \bar \a z)}.
            \end{equation*}
            Thus,
            \begin{equation*}
                \frac{\mu'(\xi)}{\mu(\xi)} \xi = \frac{|\a|^2 - 1}{(\a - \xi)(1 - \bar \a \xi)} \xi = \frac{|\a|^2 - 1}{(\a - \xi)(\bar \xi - \bar \a)} = \frac{1 - |\a|^2}{|\a - \xi|^2}, 
            \end{equation*}
            which finishes the proof.
        \end{proof}
    \subsection{Proof of Theorem \ref{thrm: multipliers to bilip}} \label{subsec: Theorem 3}
        Now we are ready to prove Theorem \ref{thrm: multipliers to bilip}:
        \begin{thrmrestate}
            Suppose $V$, $W$ are analytic discs attached to the unit sphere, and $f$, $g$ are the respective embedding maps. If $\M_f = \M_g$ with equivalent norms, then for $\xi, \z \in \TT$ such that $f(\xi) = f(\z)$ (and so Theorem \ref{thrm: algebras to crossings} implies $g(\xi) = g(\z)$) we have
            \begin{equation*}
                \frac{A_f(\xi)}{A_f(\z)} = \frac{A_g(\xi)}{A_g(\z)}.
            \end{equation*}            
        \end{thrmrestate}
        Let us introduce a notation we use in the proof. For two quantities $A,B > 0$ depending on some parameters, we write $A \asymp B$ whenever there exists a constant $C > 0$ that is independent of the parameters such that $A \le CB$ and $B \le CA$.
        \begin{proof}
            In light of Lemma \ref{lem: A_f under mu}, it is enough to prove this theorem for $\xi = 1, \z = -1$.

            Consider a metric on $\D$ induced by $\H_f$, see \cite[Lemma 9.9]{AglBook}:
            \begin{equation*}
                d_f(z,w) = \sqrt{1 - \frac{|\k_f(z,w)|^2}{\k_f(z,z)\k_f(w,w)}}.
            \end{equation*}
            It is easy to see that since $\H_f$ is complete Pick, we have
            \begin{equation*}
                d_f(z,w) = \sup \left\{ |\f(z)|: \: \f \in \M_f, \, ||\f||_{\M_f} \le 1, \, \f(w) = 0 \right\}.
            \end{equation*}
            We have a similar metric $d_g$ for $g$ instead of $f$. Since $\M_f = \M_g$ with equivalent norms, we conclude that the metrics $d_f$ and $d_g$ are equivalent, meaning that the identity map between $(\D, d_f)$ to $(\D, d_g)$ is bi-Lipschitz. Hence, $d_f^2$ and $d_g^2$ are equivalent as well, i.e.,
            \begin{equation*}
                1 - \frac{(1 - ||f(z)||^2)(1 - ||f(w)||^2)}{|1 - \langle f(z), f(w) \rangle|^2} \asymp 1 - \frac{(1 - ||g(z)||^2)(1 - ||g(w)||^2)}{|1 - \langle g(z), g(w) \rangle|^2}, \quad z, w \in \D.
            \end{equation*}
            Similarly to Theorem \ref{thrm: 1 = -1} we consider $z = 1 - x, w = -1 + y$ for $x, y > 0$ with $x, y \to 0$. Repeating the insight from Theorem \ref{thrm: 1 = -1} we set
            \begin{equation*}
                x = \frac{t}{A_f(1)}, \qquad y = \frac{t}{A_f(-1)},
            \end{equation*}
            for $t > 0$, $t \to 0$. This way, we have
            \begin{equation*}
                1 - ||f(z)||^2 = 2t + o(t), \quad 1 - ||f(w)||^2 = 2t + o(t), \quad 1 - \langle f(z), f(w) \rangle = 2t + o(t). 
            \end{equation*}
            Thus,
            \begin{equation*}
                \frac{(1 - ||f(z)||^2)(1 - ||f(w)||^2)}{|1 - \langle f(z), f(w) \rangle|^2} = \frac{(2t + o(t))(2t + o(t))}{|2t + o(t)|^2} = 1 + o(1), \quad t \to 0.    
            \end{equation*}
            We conclude that 
            \begin{equation*}
                d^2_f \left( 1 - \frac{t}{A_f(1)}, -1 + \frac{t}{A_f(-1)} \right) = o(1), \quad t \to 0. 
            \end{equation*} 
            Next, we look at what happens to $d_g^2$. Set $a = \frac{A_g(1)}{A_f(1)}$, $b = \frac{A_g(-1)}{A_f(-1)}$. We need to prove that $a = b$. We expand
            \begin{equation*}
                1 - ||g(z)||^2 = 2 a t + o(t), \quad 1 - ||g(w)||^2 = 2 b t + o(t), \quad 1 - \langle g(z), g(w) \rangle = \left( a + b \right) t + o(t).
            \end{equation*}
            Thus,  
            \begin{equation*}
                \frac{(1 - ||g(z)||^2)((1 - ||g(w)||^2))}{|1 - \langle g(z), g(w) \rangle|^2} = \frac{4ab}{(a + b)^2} + o(1), 
            \end{equation*}
            and so
            \begin{equation*}
                d^2_g \left( 1 - \frac{t}{A_f(1)}, -1 + \frac{t}{A_f(-1)} \right) = 1 - \frac{4ab}{(a + b)^2} + o(1), \quad t \to 0.
            \end{equation*}
            Since $d_g^2 \asymp d_f^2 = o(1), \ t \to 0$, we must have
            \begin{equation*}
                \frac{4ab}{(a + b)^2} = 1.
            \end{equation*}
            It remains to notice that
            \begin{equation*}
                \frac{4ab}{(a + b)^2} = 1 \iff (a + b)^2 = 4ab \iff (a - b)^2 = 0 \iff a = b.
            \end{equation*}
            We conclude that $a = b$, which means $\frac{A_f(1)}{A_g(1)} = \frac{A_f(-1)}{A_g(-1)}$, as we wanted.
        \end{proof}

    \subsection{Proof of Theorem \ref{thrm: multiplier isomorphism two points to two auto}} \label{subsec: Theorem 4}
    \begin{thrmrestate}
            Suppose $V$, $W$ are analytic discs attached to the unit sphere, and $f$, $g$ are the respective embedding maps with the only self-crossing at $\pm 1$. Then $\M_f \cong \M_g$ algebraically if and only if $\M_f = \M_{g \circ \mu}$, where either
            \begin{equation*}
                \mu(z) = \frac{z - \a}{1 - \a z}, \text{ or} \qquad \mu(z) = \frac{\b - z}{1 - \b z}
            \end{equation*}
            with constants $\a, \b \in (-1, 1)$ that can be explicitly computed in terms of $A_f(\pm 1)$ and $A_g(\pm 1)$:
            \begin{equation*}
                \a = \frac{\sqrt{A_f(1)A_g(-1)} - \sqrt{A_f(-1)A_g(1)}}{\sqrt{A_f(1)A_g(-1)} + \sqrt{A_f(-1)A_g(1)}}, \qquad \b = \frac{\sqrt{A_f(1)A_g(1)} - \sqrt{A_f(-1)A_g(-1)}}{\sqrt{A_f(1)A_g(1)} + \sqrt{A_f(-1)A_g(-1)}}.
            \end{equation*}
    \end{thrmrestate}
    \begin{proof}
        The fact that $\M_f = \M_{g \circ \mu}$ implies $\M_f \cong \M_g$ algebraically is evident for any automorphism $\mu$. Thus, we prove the other implication.

        Suppose $\M_f \cong \M_g$ algebraically. By Theorem \ref{thrm: isomorphic algebras to automorphism}, it means that there is a disc automorphism $\mu$ such that $\M_f = \M_{g \circ \mu}$ with equivalent norms. We claim that there are only two possible choices for $\mu$ as in the statement of the theorem.

        First, if $\M_f = \M_{g \circ \mu}$, then, by Theorem \ref{thrm: algebras to crossings}, $f$ and $g \circ \mu$ must have the same self-crossings on the boundary. But the only self-crossings for $f$ and $g$ are $\pm 1$, which means that $\mu(1) = \pm 1$ and $\mu(-1) = \mp 1$. Hence, either $\mu$ or $-\mu$ maps $1$ to $1$ and $-1$ to $-1$. Thus, either
        \begin{equation*}
            \mu(z) = \frac{z - \a}{1 - \a z}, \quad \a \in (-1,1) 
        \end{equation*}
        or
        \begin{equation*}
            \mu(z) = \frac{\b - z}{1 - \b z}, \quad \b \in (-1,1).    
        \end{equation*}
        It remains to prove that $\a$ and $\b$ are allowed to take only one value each, as in the statement of the theorem. Indeed, if $\M_f = \M_{g \circ \mu}$, then in light of Theorem \ref{thrm: multipliers to bilip}
        \begin{equation} \label{eq: A/A = A/A}
            \frac{A_f(1)}{A_f(-1)} = \frac{A_{g \circ \mu}(1)}{A_{g \circ \mu}(-1)}.
        \end{equation}
        
        If we consider $\mu(z) = (z - \a)/(1 - \a z)$, then \eqref{eq: A/A = A/A} together with Lemma \ref{lem: A_f under mu} gives us
        \begin{equation*}
            \frac{A_f(1)}{A_f(-1)} = \frac{A_g(1)}{A_g(-1)} \left( \frac{\a + 1}{1 - \a} \right)^2,
        \end{equation*}
        so that
        \begin{equation} \label{eq: alpha}
            \a = \frac{\sqrt{A_f(1)A_g(-1)} - \sqrt{A_f(-1)A_g(1)}}{\sqrt{A_f(1)A_g(-1)} + \sqrt{A_f(-1)A_g(1)}}.
        \end{equation}
        
        Similarly, if we consider $\mu(z) = (\b - z)/(1 - \b z)$, then
        \begin{equation*}
            \frac{A_f(1)}{A_f(-1)} = \frac{A_g(-1)}{A_g(1)} \left( \frac{\b + 1}{1 - \b} \right)^2,            
        \end{equation*}
        so that
        \begin{equation} \label{eq: beta}
            \b = \frac{\sqrt{A_f(1)A_g(1)} - \sqrt{A_f(-1)A_g(-1)}}{\sqrt{A_f(1)A_g(1)} + \sqrt{A_f(-1)A_g(-1)}}.
        \end{equation}    
    \end{proof}    

    \subsection{Proof of Theorem \ref{thrm: f_r and f_s}} \label{subsec: Theorem 5}
    \begin{thrmrestate}
        $\M_{f_r} \ne \M_{f_s}$ for $r \ne s$.
    \end{thrmrestate}
    \begin{proof}
        In light of Theorem \ref{thrm: multipliers to bilip} it is sufficient to prove that for $r > s > 0$
            \begin{equation*}
                \frac{A_{f_r}(1)}{A_{f_r}(-1)} \ne \frac{A_{f_s}(1)}{A_{f_s}(-1)}.
            \end{equation*}
            Evaluating at $\pm 1$ we get $f_r(1) = f_r(-1) = \frac{1}{\sqrt{2}}(1,1)$ and
            \begin{equation*}
                f_r'(1) = \frac{1}{\sqrt{2}} \left( 2, \, 2\frac{1 + r}{1 - r} \right), \qquad f_r'(-1) = -\frac{1}{\sqrt{2}} \left( 2, \, 2\frac{1 - r}{1 + r} \right).
            \end{equation*}
            Hence,
            \begin{equation*}
                \frac{A_{f_r}(1)}{A_{f_r}(-1)} = \frac{1 + \frac{1 + r}{1 - r}}{1 + \frac{1 - r}{1 + r}} = \frac{1 + r}{1 - r}
            \end{equation*}
            is strictly increasing for $r \in (0, 1)$, which finishes the proof.
    \end{proof}
            
    \subsection{Proof of Theorem \ref{thrm: f_r-r and f_s-s}} \label{subsec: Theorem 6}
    \begin{thrmrestate}
        $\M_{f_{r, -r}} \cong \M_{f_{s, -s}}$ algebraically if and only if  $\M_{f_{r, -r}} = \M_{f_{s, -s}}$.
    \end{thrmrestate}  
    \begin{proof}
        Note that 
            \begin{equation} \label{eq: f_r z to -z}
                f_{r, -r}(-z) = f_{-r, r}(z) = \frac{1}{\sqrt{2}}(b_{-r}(z)^2, b_r(z)^2),    
            \end{equation}
            so that if $f_{r, -r}'(1) = (z_1, z_2)$, then $f'_{r, -r}(-1) = -(z_2, z_1)$. 
            Since
            \begin{equation*}
                f_{r, -r}(1) = f_{r, -r}(-1) = \frac{1}{\sqrt{2}}(1, 1),    
            \end{equation*} 
            we get 
            \begin{equation} \label{eq: A(1) = A(-1)}
                A_{f_{r, -r}}(1) = A_{f_{r, -r}}(-1).     
            \end{equation}
        
            Let us apply Theorem \ref{thrm: multiplier isomorphism two points to two auto} to $f_{r, -r}$ and $f_{s, -s}$. Calculating $\a$, $\b$ from \eqref{eq: alpha}, \eqref{eq: beta} and using \eqref{eq: A(1) = A(-1)} we get $\a = \b = 0$. Thus, in light of \eqref{eq: f_r z to -z}, Theorem \ref{thrm: multiplier isomorphism two points to two auto} implies that $\M_{f_{r, -r}} \cong \M_{f_{s, -s}}$ algebraically if and only if either $\M_{f_{r, -r}} = \M_{f_{s, -s}}$ or $\M_{f_{r, -r}} = \M_{f_{-s, s}}$. 
            
            Note that because of \eqref{eq: f_r z to -z}
            \begin{equation*}
                \langle f_{r, -r}(z), f_{r, -r}(w) \rangle = \langle f_{-r, r}(z), f_{-r, r}(w) \rangle,
            \end{equation*}
            whence,
            \begin{equation*}
                \k^{f_{r, -r}} = \k^{f_{-r, r}},
            \end{equation*}
            meaning $\H_{f_{r, -r}} = \H_{f_{-r, r}}$ with the same inner product, which means $\M_{f_{r, -r}} = \M_{f_{-r, r}}$ isometrically.

            We conclude that $\M_{f_{r, -r}} \cong \M_{f_{s, -s}}$ algebraically if and only if $\M_{f_{r, -r}} = \M_{f_{s, -s}}$.    
    \end{proof}

    \subsection{Proof of Theorem \ref{thrm: example single 3-self-crossing}} \label{subsec: Theorem 7}
        \begin{thrmrestate}
            There is a continuum of analytic discs attached to the unit sphere with the same $3$-point self-crossing and such that their multiplier algebras are all mutually non-isomorphic.  
        \end{thrmrestate} 
        \begin{proof}
        Let us denote $1, \o, \o^2$ the third roots of unity. We consider embedding maps $f$ with a single $3$-point self-crossing $f(\o) = f(\o^2) = f(1)$. The basic building block for our embedding maps is a special type of the Blaschke product.
        \begin{lem}
            The Blaschke product 
            \begin{equation} \label{eq: g alpha definition}
                g_{\a}(z) = z \frac{z - \a}{1 - \a z} \frac{z - \b}{1 - \b z}
            \end{equation}
            with $\a \in (-1, 1)$ and
            \begin{equation*}
                \b = \frac{-\a}{1 + \a}
            \end{equation*}
            satisfies 
            \begin{equation*}
                g_{\a}(1) = g_{\a}(\o) = g_{\a}(\o^2) = 1,
            \end{equation*}
            as well as the symmetry condition $g^* = g$, where $g^*(z) = \overline{g(\bar z)}$.
        \end{lem}
        \begin{proof}
            The symmetry condition follows from the fact that $\a$ and $\b$ are real. In particular, $g_{\a}(\o^2) = g_{\a}(\bar \o) = \overline{g_{\a}(\o)}$. Note that $g_{\a}(1) = 1$ is evident, hence, it is enough to show $g_{\a}(\o) = 1$. The last statement follows by direct evaluation.
        \end{proof}

        We consider the embedding maps of the form
        \begin{equation} \label{eq: f alpha definition}
            f(z) = \frac{1}{2} \left( z^3, g_{\a_0}(z), g_{\a_1}(z), g_{\a_2}(z) \right).
        \end{equation}
        First, we need to guarantee that $f$ is indeed an embedding map of an analytic disc attached to the unit sphere. Note that $f$ is analytic in a larger disc and $||f(x)|| = 1 \iff |x| = 1$. Hence, we need to have $f'(z) \ne 0, \, z \in \D$, injectivity in $\D$ as well as $V = f(\D)$ needs to be a variety. For this purpose we fix $\a_0 \in (0, 1)$, now $f'(z) \ne 0, \, z \in \D$ clearly holds. If we have injectivity as well, then $V$ is automatically a variety as $f$ is analytic in a larger disc, see the proof of \cite[Lemma 6.2]{Kerr}. Thus, the goal is to find $\a_1$ such that $\left( z^3, g_{\a_0}(z), g_{\a_1}(z) \right)$ is injective in $\D$, this way we are still free to choose any $\a_2$. 

        \begin{lem} \label{lem: f alpha injective}
            There exists $\a_1$ such that $\left( z^3, g_{\a_0}(z), g_{\a_1}(z) \right)$ is injective in $\D$.
        \end{lem}
        \begin{proof}
            Suppose we have $z \ne w \in \D$ such that 
            \begin{equation*}
                \left( z^3, g_{\a_0}(z), g_{\a_1}(z) \right) = \left( w^3, g_{\a_0}(w), g_{\a_1}(w) \right).    
            \end{equation*} Since $z^3 = w^3$ and $z \ne w$, then $z, w \ne 0$ and either $w = \o z$ or $z = \o w$. By swapping $z$ and $w$ we can always assume that $w = \o z \ne 0$. Hence, we obtain
            \begin{equation} \label{eq: injectivity z g g}
                \left( z^3, g_{\a_0}(z), g_{\a_1}(z) \right) = \left( z^3, g_{\a_0}(\o z), g_{\a_1}(\o z) \right).   
            \end{equation}
            The idea is to take $\a_1$ in such a way that the equation \eqref{eq: injectivity z g g} does not have any solutions in $\D \setminus \{ 0 \}$, meaning that the assumption we made was wrong and $\left( z^3, g_{\a_0}(z), g_{\a_1}(z) \right)$ is injective in $\D$.
            
            Consider the equation $g_{\a_0}(z) = g_{\a_0}(\o z)$. Note that since $g_{\a_0}$ is analytic in a larger disc there can only be finitely many solutions to this equation in $\D \setminus \{ 0 \}$. If there are no solutions, then \eqref{eq: injectivity z g g} also does not have any and we are free to choose any $\a_1 \in (-1, 1)$. If there are solutions, let us denote them $z_1, \ldots, z_N \in \D \setminus \{ 0 \}$. This way $z \in \D \setminus \{ 0 \}$ is a solution for \eqref{eq: injectivity z g g} if and only if $z = z_j$ for some $j = 1, \ldots, N$ and $g_{\a_1}(z_j) = g_{\a_1}(\o z_j)$. To finish the proof let us choose $\a_1$ such that $g_{\a_1}(z_j) \ne g_{\a_1}(\o z_j)$ for any $j = 1, \ldots, N$. For a fixed $j = 1, \ldots, N$ consider a function of $\a_1$
            \begin{equation*}
                h_j(\a_1) = g_{\a_1}(z_j) - g_{\a_1}(\o z_j).
            \end{equation*}
            By \eqref{eq: g alpha definition} it is a rational function in $\a_1$. By plugging $\a_1 = 1$ we see that $h_j$ is not identically zero, as the solutions of $g_1(z) = g_1(\o z)$ are $z = 0$ and $z = \o$ which do not lie in $\D \setminus \{ 0 \}$. This means that there are finitely many values of $\a_1 \in (-1, 1)$ such that $g_{\a_1}(z_j) = g_{\a_1}(\o z_j)$, those are the values for $\a_1$ we do not want to take. Repeating this argument for all $j = 1, \ldots, N$ we remove finitely many points from the interval $(-1, 1)$ and choose $\a_1$ from the resulting set. This guarantees $g_{\a_1}(z_j) \ne g_{\a_1}(\o z_j)$ for any $j = 1, \ldots, N$. Thus, with this value of $\a_1$ the map $\left( z^3, g_{\a_0}(z), g_{\a_1}(z) \right)$ is injective in $\D$.     
        \end{proof}
        We fix a value of $\a_1$ from Lemma \ref{lem: f alpha injective}, and, since $\a_2$ is a free parameter, we write the embedding functions as
        \begin{equation} \label{eq: f alpha with fixed alpha12}
            f_{\a}(z) = \frac{1}{2} \left( z^3, g_{\a_0}(z), g_{\a_1}(z), g_{\a}(z) \right), \quad \a \in (-1, 1).
        \end{equation}

        Let us now show that as $\a$ approaches $1$ from below we get uncountably many analytic discs $f_{\a}(\D)$ such that all their multiplier algebras $\M_{f_{\a}}$ are mutually non-isomorphic. By Theorem \ref{thrm: isomorphic algebras to automorphism} we know that if they are isomorphic, then the isomorphism is given by a composition with a unit disc automorphism $\mu$. By Theorem \ref{thrm: algebras to crossings}, $\mu$ must map the self-crossing ${1, \o, \o^2}$ onto itself.
        \begin{lem} \label{lem: third roots of unity automorphisms}
            Suppose $\mu$ is a unit disc automorphism that maps $\{ 1, \o, \o^2 \}$ onto itself, then either $\mu(z) = z$, $\mu(z) = \o z$ or $\mu(z) = \o^2 z$.
        \end{lem}
        \begin{proof}
            It is enough to show that $\mu(1) = 1$ implies $\mu(z) = z$, since we can always compose $\mu$ with an appropriate rotation. Now, since $\mu$ is a disc automorphism it must preserve the order $1, \o, \o^2$ on the circle while mapping $\{ \o, \o^2 \}$ onto itself. This is only possible if $\mu(\o) = \o$ and $\mu(\o^2) = \o^2$, which, together with $\mu(1) = 1$, gives us $\mu(z) = z$.
        \end{proof}
        Now we know that the only candidates for isomorphism maps are compositions with trivial rotations. Finally, the main instrument we use is Theorem \ref{thrm: multipliers to bilip}, or, specifically, Corollary \ref{crl: n point A A A}.
        For the maps $f, g$ with a single $3$-point self-crossing at ${1, \o, \o^2}$ it reduces to the following: $\M_f = \M_g$ implies
        \begin{equation} \label{eq: 3-self-crossing A A A}
            [A_f(1): A_f(\o): A_f(\o^2)] = [A_g(1): A_g(\o): A_g(\o^2)].
        \end{equation}
        Together with Theorem \ref{thrm: isomorphic algebras to automorphism}, Theorem \ref{thrm: algebras to crossings} and Lemma \ref{lem: third roots of unity automorphisms} we obtain that $\M_f \cong \M_g$ implies the identity \eqref{eq: 3-self-crossing A A A} up to composing $g$ with $\mu(z) = z$, $\mu(z) = \o z$ or $\mu(z) = \o^2 z$. In light of Lemma \ref{lem: A_f under mu} this composition reduces to the cyclic permutation of the expression, i.e., at least on of the three identities
        \begin{align*}
            [A_f(1): A_f(\o): A_f(\o^2)] &= [A_g(1): A_g(\o): A_g(\o^2)], \\
            [A_f(1): A_f(\o): A_f(\o^2)] &= [A_g(\o^2): A_g(1): A_g(\o)], \\
            [A_f(1): A_f(\o): A_f(\o^2)] &= [A_g(\o): A_g(\o^2): A_g(1)]
        \end{align*}
        holds.
        
        For our embedding maps $f_{\a}$ this takes a simpler form. Indeed, by definition \eqref{eq: f alpha definition}, we have the symmetry $f_{\a}(z) = \overline{f_{\a}(\bar z)}$, which means that $A_{f_{\a}}(\o) = A_{f_{\a}}(\o^2)$, since $\o^2 = \bar \o$. Thus, if $\M_{f_{\a}} \cong \M_{f_{\tilde \a}}$, then 
        \begin{equation*}
            [A_{f_{\a}}(1): A_{f_{\a}}(\o): A_{f_{\a}}(\o^2)] = [A_{f_{\tilde \a}}(1): A_{f_{\tilde \a}}(\o): A_{f_{\tilde \a}}(\o^2)]    
        \end{equation*}
        must hold without any cyclic permutations. It remains to directly compute $A_{f_{\a}}(1): A_{f_{\a}}(\o)$ for different $\a$. We will show that $A_{f_{\a}}(1): A_{f_{\a}}(\o)$ continuously tends to infinity as $\a$ approaches $1$ from below. In turn, this will imply that there are uncountably many $\a$ such that $\M_{f_{\a}}$ are all mutually non-isomorphic. 
        
        The continuity directly follows from the fact that $f_{\a}(z)$ and $f'_{\a}(z)$ continuously depend on $\a$ for a fixed $z$, so we only need to show that 
        \begin{equation*}
            \frac{A_{f_{\a}}(1)}{A_{f_{\a}}(\o)} \to \infty, \quad \a \to 1^-.     
        \end{equation*}
        In \eqref{eq: f alpha with fixed alpha12} $\a$ only affects the fourth coordinate, hence, it is enough to show that $|g_{\a}(\o) g'_{\a}(\o)|$ is bounded while $|g_{\a}(1) g'_{\a}(1)|$ goes to infinity for $\a \to 1^-$. Note that $g_{\a}(\o) = g_{\a}(1) = 1$, meaning we need to show that $\left| g'_{\a}(\o)/g_{\a}(\o) \right|$ is bounded while $\left| g'_{\a}(1)/g_{\a}(1) \right|$ goes to infinity. Directly computing the logarithmic derivative of $g_{\a}$ from \eqref{eq: g alpha definition} we get
        \begin{equation*}
            \frac{g'_{\a}(z)}{g_{\a}(z)} = \frac{1}{z} + \frac{1}{z - \a} + \frac{1}{z - \b} + \frac{\a}{1 - \a z} + \frac{\b}{1 - \b z}.
        \end{equation*}
        Note that as $\a \to 1$ we get $\b \to -1/2$, so that
        \begin{equation*}
            \frac{g'_{\a}(\o)}{g_{\a}(\o)} \to \frac{1}{\o} + \frac{1}{\o - 1} + \frac{1}{\o + 1/2} + \frac{1}{1 - \o} + \frac{-1/2}{1 + \o/2}, \quad \a \to 1,
        \end{equation*}
        in particular, it is bounded. While
        \begin{equation*}
            \frac{g'_{\a}(1)}{g_{\a}(1)} = 1 + \frac{1}{1 - \a} + \frac{1}{1 - \b} + \frac{\a}{1 - \a} + \frac{\b}{1 - \b} \to \infty, \quad \a \to 1,
        \end{equation*}
        as we wanted. We conclude that 
        \begin{equation*}
            \frac{A_{f_{\a}}(1)}{A_{f_{\a}}(\o)} \to \infty, \quad \a \to 1^-
        \end{equation*}
        continuously, which means that that there are uncountably many values of $\a$ for which the multiplier algebras $\M_{f_{\a}}$ of analytic discs $f_{\a}(\D)$ are mutually non-isomorphic. At the same time all $f_{\a}$ have exactly the same $3$-point self-crossing at $\{1, \o, \o^2\}$.
        \end{proof}

\subsection*{Acknowledgements}
    I would like to express my gratitude to Orr Shalit for his invaluable support and guidance throughout this project, which originated from my master’s thesis under his supervision. I am also thankful to Oleg Ivrii, Victor Vinnikov, and Eli Shamovich, whose comments and remarks significantly enhanced the exposition of my thesis, and by extension, this paper. Finally, I wish to thank Jeet Sampat for his thoughtful suggestions that helped improve the final manuscript.

\printbibliography 
 
\end{document}